\documentclass[11pt]{amsart}
\usepackage{amsfonts,amssymb,amscd,amsmath,enumerate,verbatim,calc,thumbpdf,stmaryrd, graphicx,ulem,mathtools}
\usepackage{tikz,hyperref}
\usepackage[all]{xy}

\theoremstyle{plain}
\newtheorem{thm}{Theorem}[section]
\newtheorem{lem}[thm]{Lemma}
\newtheorem{cor}[thm]{Corollary} 
\newtheorem{prop}[thm]{Proposition}

\theoremstyle{definition}
\newtheorem{defin}[thm]{{Definition}}
\newtheorem{ex}[thm]{Example}
\newtheorem{fact}[thm]{Fact}
\newtheorem{notation}[thm]{Notation}

\newtheorem{Assumption}[thm]{Assumption}

\numberwithin{equation}{thm}

\DeclareMathOperator{\lcm}{lcm}
\DeclareMathOperator{\mdimen}{m-dim}
\DeclareMathOperator{\distance}{dist}

\newcommand{\moninf}{\mathbb{M}^i_{\infty \geq 0}}
\newcommand{\moninfi}{\mathbb{M}^i_{\infty \geq 0}}
\newcommand{\eref}[1]{Example~\ref{e#1}}
\newcommand{\tref}[1]{Theorem~\ref{t#1}}
\newcommand{\fref}[1]{Fact~\ref{f#1}}
\newcommand{\lref}[1]{Lemma~\ref{l#1}}
\newcommand{\pref}[1]{Proposition~\ref{p#1}}

\newcommand{\cref}[1]{Corollary~\ref{c#1}}

\newcommand{\sref}[1]{Section~\ref{s#1}}
\newcommand{\ul}[1]{\underline{#1}}

\newcommand{\ideal}[3]{#1_{\underline{#2},\underline{#3}}}

\newcommand{\mon}[3]{\underline{#1}^{\underline{#2}_{#3}}}
\newcommand{\mono}[2]{\underline{#1}^{\underline{#2}}}
\newcommand{\monset}[1]{\left\llbracket #1 \right\rrbracket}
\newcommand{\start}[2]{\kern{#1em} #2 | }
\newcommand{\st}{\; | \;}
\newcommand{\dist}[2]{\distance(#1,#2)}
\newcommand{\mdim}[1]{\mdimen(#1)}

\newcommand{\moncrs}{\mathbb{M}^1_{\geq 0}\times\cdots\times\mathbb{M}^d_{\geq 0}}
\newcommand{\mongeqi}{\mathbb{M}^i_{\geq 0}}

\def\hence{%
   \;\,
   \leavevmode
   \lower0.2ex\hbox{$\cdotp$}%
   \kern0.0em\raise0.7ex\hbox{$\cdotp$}%
   \kern0.0em\lower0.2ex\hbox{$\cdotp$}%
   \thinspace
   \;\,}

\begin{document}

\begin{abstract}
We develop a new technique for studying monomial ideals in the standard polynomial rings $A[X_1,\ldots,X_d]$ where $A$ is a commutative ring with identity. The main idea is to consider induced ideals in the semigroup ring $R=A[\moncrs]$ where $\mathbb{M}^1,\ldots,\mathbb{M}^d$ are non-zero additive subgroups of $\mathbb{R}$. We prove that the set of non-zero finitely generated monomial ideals in $R$ has the structure of a metric space, and we prove that a version of Krull dimension for this setting is lower semicontinuous with respect to  this metric space structure. We also show how to use discrete techniques to study certain monomial ideals in this context.
\end{abstract}

\title[Krull dimension of monomial ideals]{Krull dimension of monomial ideals in polynomial rings with real exponents}
\author[Zechariah Andersen]{Zechariah Andersen}
\address{Department of Mathematics,
NDSU Dept \# 2750,
PO Box 6050,
Fargo, ND 58108-6050
USA}
\email{zechariah.andersen@my.ndsu.edu}
\author{Sean Sather-Wagstaff}

\email{sean.sather-wagstaff@ndsu.edu}

\urladdr{http://www.ndsu.edu/pubweb/\~{}ssatherw/}

\thanks{
Sean Sather-Wagstaff was supported in part by a grant from the NSA}


\dedicatory{To our teacher and friend, Jim Coykendall}

\maketitle

\section{Introduction}

\begin{Assumption}
Throughout this paper, $A$ is a non-zero commutative ring with identity, and
$\mathbb{M}^1, \ldots, \mathbb{M}^d$ are non-zero additive subgroups of $\mathbb{R}$.
 For $i = 1,\ldots,d$ set
 $\mathbb{M}^i_{\geq 0} = \{m \in \mathbb{M}^i \st m \geq 0\}$.
\end{Assumption}

We are interested in  properties of monomial ideals 
in the polynomial ring $S=A[X_1,\ldots,X_d]$, that is, ideals generated by monomials. 
These ideals have deep applications to combinatorics; for instance, 
building from work of Hochster~\cite{hochster:cmrcsc} and Reisner~\cite{reisner:cmqpr}, Stanley uses 
monomial ideals to prove his upper bound theorem for simplicial spheres~\cite{stanley:ubccmr}.
On the other hand, one can use the combinatorial aspects of these ideals to construct interesting 
examples and verify their properties. See, e.g., \cite{bruns:cmr, miller:cca, rogers:mid, stanley:cca, villarreal:ma} for some aspects of 
this.

For small values of $d$, one can study a given monomial ideal $I\subseteq S$ visually. 
For instance, when $d=2$, one considers the set of points $(a_1,a_2)\in\mathbb Z^2$ such that $X_1^{a_1}X_2^{a_2}\in I$.
This ``graph of $I$'' contains non-trivial information about $I$; for example, one can read certain
decompositions of $I$ from the graph. 

Given the fact that these graphs are (in general) subsets of $\mathbb R^d$, one should be able to study these ideals
using geometric techniques, as follows.
To prove a result about a given monomial ideal $I$, prove it for monomial ideals $J$ that are  ``close'' to $I$
in some suitable sense, and then prove that the closeness of $I$ to $J$ forces the conclusion to transfer from $J$ to $I$.
The problem with this idea is the following: 
the ideal $I$ is defined by discrete data (e.g., the lattice points $(a_1,a_2)$ described above).
Thus, a reasonable notion of ``closeness'' for  classical monomial ideals  based on the euclidean metric in $\mathbb R^n$ 
will likely be trivial. 
As a possible remedy for this defect, we switch perspectives from the discrete setting to a continuous one.

We consider the semigroup ring $R = A[\moncrs]$
which is the set of all 
(finite) 
$A$-linear combinations of \textit{monomials} 
$\mono{X}{m} = X_1^{m_1}X_2^{m_2}\cdots X_d^{m_d}$
where $\ul{m} = (m_1,m_2,\ldots,m_d) \in \moncrs$.
A \textit{monomial ideal} of $R$ is an ideal of $R$ that is generated by a set of monomials. 
(This includes the ideal $0=(\emptyset)R$.)
For instance, in the case $\mathbb M^i=\mathbb R$, 
the monomials of $R$ correspond to the points of the non-negative orthant $\mathbb R^d_{\geq 0}$,
and any monomial ideal  $I\subseteq S$ induces a monomial ideal  $IR\subseteq R$ since $S\subseteq R$.\footnote{It is worth noting
that the ring $A[\moncrs]$ has been studied previously, for instance, to construct interesting counterexamples to questions about non-noetherian rings;
see, e.g., \cite{anderson:mdcad}. We are grateful to Jim Coykendall for teaching us about these constructions.}
On the other hand, the case $\mathbb M^i=\mathbb Z$ recovers the monomial ideals of $S$.

The main results of this paper are in \sref{mdim} where we study a version of the Krull dimension
for this setting.
We prove that the set of non-zero finitely generated monomial ideals in $R$ has the structure of
a metric space in \tref{metric}.
(For example, in the case $\mathbb M^i=\mathbb R$, this applies to the ideals of the form $IR$ where $I$ a non-zero monomial ideal of $S$.)
Then in \cref{or130521a} we prove that our Krull dimension is lower semicontinuous with respect to  this metric space structure.
This suggests that one may be able to apply the geometric techniques described above in this setting, even to monomial ideals in $S$. 

In
\sref{ec130510a}, we run this idea in reverse, in some sense, by showing how to use discrete techniques
from $S$ to study monomial ideals in $R$ for the case $\mathbb M^i=\mathbb R$.
Specifically, we apply techniques from \cite{paulsen:eiwg} to a special class of monomial ideals of $R$ that behave like
edge ideals of weighted graphs. The main result of this section is
\tref{hm130513a} which provides non-trivial decompositions of these ideals determined by objects that we call ``interval vertex covers''. 

In a sense,
\sref{finite} consists of background material and  examples. 
On the other hand, many of the results in this section are technically new, being versions of results from~\cite{ingebretson:dmirsr}
for our more general context.

\section{Monomial Ideals and their Decompositions} \label{sirred}\label{sfinite}
This section contains foundational material for use in the rest of the paper. 
Most of the ideas are standard for the case $\mathbb{M}^i=\mathbb{Z}$,
and the case $\mathbb M^i=\mathbb R$ is developed in \cite{ingebretson:dmirsr}.

\begin{Assumption}
Throughout this section, set
$R = A[\moncrs]$. 
\end{Assumption}

\subsection*{Monomial Basics}

\begin{defin}\label{dlcm}
For $i = 1,\ldots,d$ set
$\mathbb{M}^i_{> 0} = \{m \in \mathbb{M}^i \st m > 0\}$ and $\mathbb{M}^i_{\infty \geq 0} = \mathbb{M}^i_{\geq 0} \cup \{ \infty \}$.
A \textit{pure power} in $R$ is a monomial  of the form $X_i^r$.
For any subset $G \subseteq R,$ we let
$\monset{G}$ denote the set of all monomials in $G$, so $\monset{G} = G \cap \monset{R}$.
For $i=1,\ldots,d,$ we define\footnote{Despite this notation, note that 0 is not a monomial according to our definition.} 
$X_i^\infty = 0$.
\end{defin}

The following is a straightforward consequence of our definitions.

\begin{fact}\label{f2.2}
\label{f2.3}
Fix a set $\{I_\lambda\}_{\lambda \in \Lambda}$ of monomial ideals of $R$.
 \begin{enumerate}[(a)]
  \item \label{item130521a}
  The monomial ideal $I_\lambda$ is generated by $\monset{I_\lambda}$, so we have $I_\lambda \subseteq I_\mu$ if and only if $\monset{I_\lambda} \subseteq \monset{I_\mu},$ and hence $I_\lambda=I_\mu$ if and only if $\monset{I_\lambda} = \monset{I_\mu}$.
  \item \label{item130521b}
  Given monomials $f = \mono{X}{r}$ and $g = \mono{X}{s}$ in $R,$ we have $g \in (f)R$ if and only if for all $i$ there exists $t_i \in \mongeqi$ such that $r_i + t_i = s_i$. When these conditions are satisfied, we have $g=fh$ where $h=\ul{X}^{\ul{t}}$.
  \item \label{item130521c}
  Given a monomial $f \in \monset{R}$ and a subset $S \subseteq \monset{R},$ we have $f \in (S)R$ if and only if $f \in (s)R$ for some $s \in S$. 
  \item \label{item130521d}
  The sum $\sum_{\lambda \in \Lambda} I_\lambda$
  and intersection $\bigcap_{\lambda \in \Lambda} I_\lambda$ are monomial ideals such that $\monset{\sum_{\lambda \in \Lambda}I_\lambda} = \bigcup_{\lambda \in \Lambda} \monset{I_\lambda}$ and $\monset{\bigcap_{\lambda \in \Lambda} I_\lambda} = \bigcap_{\lambda \in \Lambda} \monset{I_\lambda}$. 
 \end{enumerate}
\end{fact}

\begin{ex}\label{ex130508a}
As with monomial ideals in the polynomial ring $A[X_1,X_2]$, we can visualize a
monomial ideal $I$ in $R=A[\mathbb{R}_{\geq 0}\times\mathbb{R}_{\geq 0}]$
via $\monset{I}$. For instance, here is $\monset I$ where $I=(X_1X_2^a\mid a>1)R$.

\

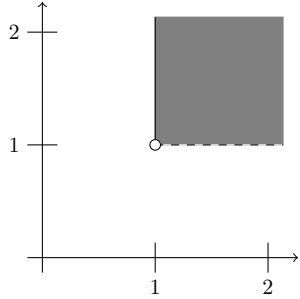
\begin{center}
\begin{tikzpicture}
	\draw[->] (-0.2,0) -- (3.4,0);
	\draw[->] (0,-0.2) -- (0,3.4);
	\draw (1.5,-0.2) node[below,scale=.75]{$1$} -- (1.5,0.2);
	\draw (3,-0.2) node[below,scale=.75]{$2$} -- (3,0.2);
	\draw (-0.2,1.5) node[left,scale=.75]{$1$} -- (0.2,1.5);
	\draw (-0.2,3) node[left,scale=.75]{$2$} -- (0.2,3);
	\draw[fill,color=black!50] (1.5,1.5) -- (3.2,1.5) -- (3.2,3.2) -- (1.5,3.2) -- cycle;
	\draw (1.5,1.5) -- (1.5,3.2);
	\draw[color=white] (1.5,1.5) -- (3.2,1.5);
	\draw[dashed] (1.5,1.5) -- (3.2,1.5);
	\draw (1.5,1.5)[fill,color=white] circle (2pt);
	\draw (1.5,1.5) circle (2pt);
\end{tikzpicture}
\end{center}
\end{ex}

\begin{defin}
Let $I$ be a monomial ideal of $R$ and suppose that $I = (\ul{X}^{\ul{\alpha}_1},\ldots,\ul{X}^{\ul{\alpha}_n})R$. We say that the list $\ul{X}^{\ul{\alpha}_1},\ldots,\ul{X}^{\ul{\alpha}_n}$ is an \textit{irredundant generating sequence} for $I$ if for all $i\neq j$ we have $\ul{X}^{\ul{\alpha}_i} \notin (\ul{X}^{\ul{\alpha}_j})R$.
\end{defin}

\begin{fact}
As a consequence of \fref{2.2},
one checks readily that every finitely generated monomial ideal in $R$ has a unique irredundant monomial generating sequence.
(Note, however, that $R$ may have monomial ideals that are not finitely generated.)
\end{fact}

 The next result is proved as in \cite[Lemma~2.7]{ingebretson:dmirsr}, using \fref{2.3}.

\begin{lem}\label{l2.7} 
 For $t = 1,\ldots,l,$ let $\{K_{t,i_t}\}_{i_t=1}^{m_t}$ be a collection of monomial ideals of $R$. Then the following equalities hold:
 \[ \bigcap_{t=1}^l \sum_{i_t = 1}^{m_t} K_{t,i_t} = \sum_{i_1=1}^{m_1} \sum_{i_2 = 1}^{m_2} \cdots \sum_{i_l = 1}^{m_l} \bigcap_{t=1}^l K_{t,i_t}
        \]
\[ \sum_{t=1}^l \bigcap_{i_t = 1}^{m_t} K_{t,i_t} = \bigcap_{i_1=1}^{m_1}\bigcap_{i_2=1}^{m_2}\cdots \bigcap_{i_l = 1}^{m_l} \sum_{t=1}^l K_{t,i_t} \]
\end{lem}

\subsection*{Generators of Intersections of Monomial Ideals}

\

\noindent
\fref{2.2}\eqref{item130521d} shows that an intersection of monomial ideals is again a monomial ideal. 
In this subsection, we explicitly describe a monomial generating sequence for any finite intersection of
monomial ideals; see \pref{2.5}. This is key for many of our results, and it strongly relies on the assumption
that each $\mathbb{M}^i$ is closed under subtraction.

\begin{defin}\label{dlcmA}
Let $\mon{X}{r}{1}, \ldots, \mon{X}{r}{k} \in \monset{R}$ with $\ul{r}_i = (r_{i,1},\ldots,r_{i,d})$. The \textit{least common multiple} of the $\mon Xri$ is
$\lcm_{\begin{subarray}{1}1 \leq i \leq k\end{subarray}}(\mon{X}{r}{i}) = \mono{X}{p}$
 where $\ul{p}$ is defined as
 \[p_j = \inf\{ m \in \mongeqi \st r_{i,j} + t_i = m\ \text{ for some } t_i \in \mongeqi \text{ and } i=1,\ldots,d\}.\]
\end{defin}

\begin{lem}\label{llcmsub}
Given $\mon{X}{r}{1},\ldots,\mon{X}{r}{k} \in \monset R$, we have that
$\lcm_{1 \leq i \leq k}(\mon{X}{r}{i}) = \mono{X}{p}$
 where $p_i = \max_{1 \leq j \leq k} \{r_{i,j}\}$.
 \begin{proof}
We prove the case $k=2$; the general case is handled similarly.
  Fix $\mono{X}{r}$ and $\mono{X}{q} \in \monset R$, and for $i=1,\ldots,d,$ set 
  \[L_i = \{m \in \mongeqi \st r_i + \alpha = m =q_i + \beta  \text{ for some } \alpha,\beta \in \mongeqi\}.\]
  We can rewrite each $L_i$ as
  \begin{align*}
   L_i 
       &= \{m \in \mongeqi \st m-r_i \geq 0 \text{ and } m-q_i \geq 0 \} \\
       &= \{m \in \mongeqi \st m \geq r_i \text{ and } m \geq q_i \} \\
       &= \{m \in \mongeqi \st m \geq \max\{r_i,q_i\} \}.
  \end{align*}
The first  equality follows from the fact that $\mathbb{M}^i$ is closed under subtraction,  and the other equalities are straightforward. Thus, by definition, 
we have $\lcm(\mono{X}{r},\mono{X}{q}) = \mono{X}{s},$ where $s_i = \inf(L_i) = \max\{r_i,q_i\}$.
 \end{proof}
\end{lem}

The next result is proved like~\cite[Theorem 2.5]{ingebretson:dmirsr}, using \lref{lcmsub}.

\begin{prop}\label{p2.5}
Given subsets $S_1,\ldots, S_k \subseteq \monset{R},$ we have
 \[ \bigcap_{i=1}^k(S_i)R = \left(\left\{\lcm_{1\leq i \leq k} (f_i) \st f_i \in S_i \text{ for } i = 1,\ldots, k \right\}\right)R.\]
\end{prop}

\subsection*{M-Irreducible  Monomial Ideals}

\

\noindent Here we characterize the monomial ideals that cannot be decomposed as a non-trivial, finite intersection of monomial ideals;
see \pref{classif1}.

\begin{notation}
 Let $\varepsilon \in \{0,1\}$. Given $r \in \mathbb{M}^i$ and $\alpha \in \mathbb{R},$ we define \[ r \geq_\varepsilon \alpha \textnormal{ provided that } 
 \begin{cases}
  r \geq \alpha & \text{if $\varepsilon = 0$} \\
  r > \alpha & \text{if $\varepsilon = 1$}.
 \end{cases}
\]
Given $s \in \moninf,$ we define 
\[s \geq_{\varepsilon} \infty \text{ provided that } s = \infty.\]
Given $\ul{\alpha} = (\alpha_1,\ldots,\alpha_d) \in \mathbb{R}_{\infty \geq 0}^d$ and $\ul{\varepsilon} = (\varepsilon_1,\ldots,\varepsilon_d) \in \{0,1\}^d,$ we set
\[ J_{\ul{\alpha},\ul{\varepsilon}} = (\{ X_i^{r_i} \st r_i \in \moninfi \text{ and } r_i \geq_{\varepsilon_i} \alpha_i \text{ for } i = 1,\ldots,d\})R. \]
Note that $J_{\ul{\alpha},\ul{\varepsilon}}$ is generated by pure powers in $R$.
\end{notation}

\begin{ex}\label{ex130508g}
In $R=A[\mathbb{R}_{\geq 0}\times\mathbb{R}_{\geq 0}]$
we illustrate the ideals $J_{(1,1),(0,1)}$ and $J_{(1,1),(0,0)}=(X_1,X_2)R$.

\

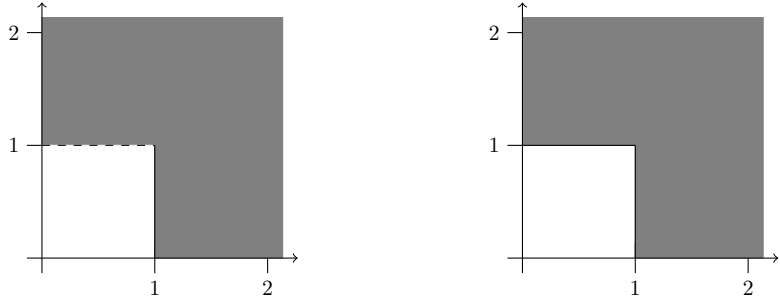
\begin{center}
\begin{tikzpicture}
	\draw (1.5,-0.2) node[below,scale=.75]{$1$} -- (1.5,0.2);
	\draw (3,-0.2) node[below,scale=.75]{$2$} -- (3,0.2);
	\draw (-0.2,1.5) node[left,scale=.75]{$1$} -- (0.2,1.5);
	\draw (-0.2,3) node[left,scale=.75]{$2$} -- (0.2,3);
	\draw[fill,color=black!50] (0,1.5) -- (1.5,1.5) -- (1.5,0) -- (3.2,0) -- (3.2,3.2) -- (0,3.2) -- cycle;
	\draw[->] (-0.2,0) -- (3.4,0);
	\draw[->] (0,-0.2) -- (0,3.4);
	\draw[color=white] (0,1.5) -- (1.5,1.5);
	\draw[dashed] (0,1.5) -- (1.5,1.5);
	\draw (1.5,0) -- (1.5,1.5);
\end{tikzpicture}
\qquad
\qquad
\qquad
\begin{tikzpicture}
	\draw (1.5,-0.2) node[below,scale=.75]{$1$} -- (1.5,0.2);
	\draw (3,-0.2) node[below,scale=.75]{$2$} -- (3,0.2);
	\draw (-0.2,1.5) node[left,scale=.75]{$1$} -- (0.2,1.5);
	\draw (-0.2,3) node[left,scale=.75]{$2$} -- (0.2,3);
	\draw[fill,color=black!50] (0,1.5) -- (1.5,1.5) -- (1.5,0) -- (3.2,0) -- (3.2,3.2) -- (0,3.2) -- cycle;
	\draw[->] (-0.2,0) -- (3.4,0);
	\draw[->] (0,-0.2) -- (0,3.4);
	\draw (0,1.5) -- (1.5,1.5);
	\draw (1.5,0) -- (1.5,1.5);
\end{tikzpicture}
\end{center}
\end{ex}

\begin{defin}\label{d3.7}
 A monomial ideal $I \subseteq R$ is \textit{m-irreducible} (short for \textit{mono-mial-irreducible}) provided that for all monomial ideals $J$ and $K$ of $R$ such that $I = J \cap K,$ either $I = J$ or $I = K$.
\end{defin}

The following characterization of m-irreducible monomial ideals is proved as in \cite[Theorem 3.9]{ingebretson:dmirsr} using \fref{2.2} and \pref{2.5}.

\begin{prop}\label{pclassif1} 
For a monomial ideal $I \subseteq R$, the following conditions are equivalent.
 \begin{enumerate}[\rm(i)]
  \item $I$ is generated by pure powers of a subset of the variables $X_1,\ldots,X_d$.
  \item there exist $\ul{\alpha} \in \mathbb{R}_{\infty \geq 0}^d$ and $\ul{\varepsilon} \in \{0,1\}^d$ such that $I = J_{\ul{\alpha},\ul{\varepsilon}}$.
  \item $I$ is m-irreducible.
 \end{enumerate}
 \end{prop}

\begin{ex}\label{ex130508i}
\pref{classif1} implies that the ideals in Example~\ref{ex130508g} are m-irreducible. 
It is worth noting that, even though the following graph has roughly the same shape as those in Example~\ref{ex130508g}

\

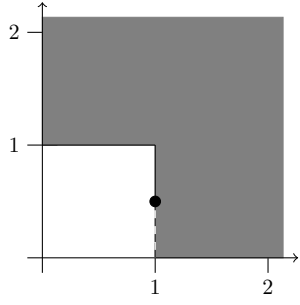
\begin{center}
\begin{tikzpicture}
	\draw (1.5,-0.2) node[below,scale=.75]{$1$} -- (1.5,0.2);
	\draw (3,-0.2) node[below,scale=.75]{$2$} -- (3,0.2);
	\draw (-0.2,1.5) node[left,scale=.75]{$1$} -- (0.2,1.5);
	\draw (-0.2,3) node[left,scale=.75]{$2$} -- (0.2,3);
	\draw[fill,color=black!50] (0,1.5) -- (1.5,1.5) -- (1.5,0) -- (3.2,0) -- (3.2,3.2) -- (0,3.2) -- cycle;
	\draw[->] (-0.2,0) -- (3.4,0);
	\draw[->] (0,-0.2) -- (0,3.4);
	\draw[color=white] (1.5,0) -- (1.5,0.75);
	\draw[dashed] (1.5,0) -- (1.5,0.75);
	\draw (1.5,0.75) -- (1.5,1.5);
	\draw (0,1.5) -- (1.5,1.5);
	\draw (1.5,0.75)[fill,color=black] circle (2pt);
\end{tikzpicture}
\end{center}
the ideal 
$I=(X_1^a,X_2,X_1X_2^{1/2}\mid a>1)R$
it represents
in $R=A[\mathbb{R}_{\geq 0}\times\mathbb{R}_{\geq 0}]$
is not m-irreducible.
This follows from \pref{classif1} because $I$ cannot be generated by pure powers of the variables.
One also deduces this by definition using the decomposition
$$I=(X_1,X_2)R\cap(X_1^a,X_2^{1/2}\mid a>1)R=J_{(1,1),(0,0)}\cap J_{(1,1/2),(1,0)}.$$
This decomposition is non-trivial, since \fref{2.3}
implies that we have $X_1\in J_{(1,1),(0,0)}\smallsetminus J_{(1,1/2),(1,0)}$
and  $X_2^{1/2}\in J_{(1,1/2),(1,0)}\smallsetminus J_{(1,1),(0,0)}$.
\end{ex}

\subsection*{M-Prime Monomial Ideals}

\

\noindent
Here we characterize the ideals of $R$ that are prime with respect to monomials, for use in \sref{mdim}.

\begin{defin}\label{dmprime}
A monomial ideal $P \subsetneq R$ is \textit{m-prime} (short for \textit{monomial-prime}) provided that 
for all $f,g \in \monset{R},$ if $f\cdot g \in P,$ then $f \in P$ or $g \in P$.
Given a subset $T \subseteq \{1,\ldots,d\}$, set \[Q_T = (X_i^n \st i \in T \text{ and } n \in \mathbb{M}^i_{> 0})R. \]
\end{defin}

\begin{prop}\label{ppure}
For a monomial ideal $I \subseteq R$, the following conditions are equivalent.
 \begin{enumerate}[\rm(i)]
  \item $I$ is m-prime.
  \item There exists $T \subseteq \{1,\ldots,d\}$ such that $I = Q_T$.
  \item $I = J_{\ul\alpha,\ul 1}$ where $\alpha_i \in \{0,\infty\}$ for all $i$ and $\ul{1}=(1,\ldots,1)$.
 \end{enumerate}
\end{prop}
\begin{proof}
(i) $\Rightarrow$ (ii): Assume that $I$ is m-prime, and set
$$T=\{i\mid\text{$X_i^n\in I$ for some $n\in\mathbb M^i_{<0}$}\}.$$
First, observe that if $i\in T$, then $X_i^a\in I$ for all $a\in\mathbb M^i_{>0}$.
Indeed, by definition of $T$, there is an element $n\in\mathbb M^i_{>0}$ such that $X_i^n\in I$. Fix a positive integer $k$ such that $ak > n$.
Since $\mathbb{M}^i$ is closed under subtraction, we have  $ak -n \in \mathbb{M}^i_{> 0}$. 
It follows  that $X_i^{ak}=X_i^{ak - n}X_i^n  \in I$, so the fact that $I$ is m-prime implies that $X_i^a\in I$, as claimed.

Now we show that $I=Q_T$. For the containment $I\supseteq Q_T$, it suffices to show that each generator $X_i^a$ of $Q_T$ is in $I$;
here we have $i\in T$ and $a\in\mathbb M^i_{>0}$. This follows from the above observation. 
For the reverse containment $I\supseteq Q_T$, let $X_1^{\alpha_1}\cdots X_d^{\alpha_d}\in\monset I$.
Since $I$ is m-prime, there is an index $i$ such that $\alpha_i>0$ and $X_i^{\alpha_i}\in I$.
It follows that $i\in T$, so $X_1^{\alpha_1}\cdots X_d^{\alpha_d}\in(X_i^{\alpha_i})R\subseteq Q_T$.
  
  (ii) $\Rightarrow$ (iii): Let $T \subseteq \{1,\ldots,d\}$ and let $I = Q_T$. For $i=1,\ldots,d$ set
  $$\alpha_i=\begin{cases}
0 & \text{if } i \in T \\
\infty &\text{if } i \not \in T.\end{cases}
$$
It is straightforward to show that $I = Q_T=J_{\ul\alpha,\ul 1}$.
  
  (iii) $\Rightarrow$ (i): Let $I = J_{\ul\alpha,\ul 1}$ such that $\alpha_i \in \{0,\infty\}$ for $i=1,\ldots,d$. 
  To show that $I$ is m-prime, let $\ul\gamma,\ul\beta\in\moncrs$ such that
  $\mono{X}{\gamma}\mono{X}{\beta} \in I$. Since $\mono{X}{\gamma}\mono{X}{\beta}$ must be a multiple of one of the generators of $J_{\ul\alpha,\ul 1}$, there exists an $i$ such that $\alpha_i = 0$ and $\gamma_i + \beta_i > 0.$ Hence, either $\gamma_i > 0$ or $\beta_i >0.$ Suppose without loss of generality that $\gamma_i > 0$. Then $X_i^{\gamma_i} \in I,$ which implies that $\mono{X}{\gamma} \in I.$ Hence, $I$ is m-prime.
 \end{proof}

\subsection*{M-Irreducible Decompositions}

\

\noindent
Here we characterize the monomial ideals that can be expressed as finite intersections of monomial irreducible ideals; see \pref{classif2}.
\begin{defin} \label{d4.1}
Let $I \subseteq R$ be a monomial ideal. An \textit{m-irreducible decomposition} of $I$ is a decomposition $I=\bigcap_{\lambda \in \Lambda} I_\lambda$ where each $I_\lambda$ is an m-irreducible monomial ideal of $R$. If the index set $\Lambda$ is finite, we say that $I = \bigcap_{\lambda \in \Lambda} I_\lambda$ is a \textit{finite m-irreducible decomposition}. An m-irreducible decomposition is \textit{irredundant} if for all distinct $\lambda,\mu\in\Lambda$ we have $I_\lambda\not\subseteq I_\mu$.
\end{defin}

\begin{notation}
Given $\ul{\alpha} \in \mathbb{R}_{\infty \geq 0}^d$ and $\ul{\varepsilon} \in \{0,1\}^d$, we set \[ I_{\ul{\alpha},\ul{\varepsilon}} = (\{ \mono{X}{r} \st i = 1,\ldots,d \text{, } r_i \geq_{\varepsilon_i} \alpha_i \text{ and } r_i \in \moninfi \})R. \]
Note that $I_{\ul{\alpha},\ul{\varepsilon}}$ is not in general generated by pure powers of the variables, so it is different from $J_{\ul{\alpha},\ul{\varepsilon}}$.
\end{notation}

\begin{ex}\label{e4.3}
With the zero-vector $\ul{0} = (0,\ldots,0)$, we have $I_{\ul{\alpha},\ul{0}} = (\mono{X}{\alpha})R$. 
From this it follows that if $I$ is a finitely generated monomial ideal in $R$, then $I$ is a finite sum of ideals of the form $I_{\ul{\alpha},\ul{0}}$.
Indeed, let $\underline{X}^{\underline\alpha_1},\ldots,\underline{X}^{\underline\alpha_n}$ be a monomial generating sequence for $I$.
Then we have
$$I=(\underline{X}^{\underline\alpha_1},\ldots,\underline{X}^{\underline\alpha_n})R=\sum_{i=1}^n(\underline{X}^{\underline\alpha_i})R=\sum_{i=1}^nI_{\ul{\alpha}_i,\ul{0}}$$
On the other hand, if $\alpha_i = \infty$ for any $i,$ then $I_{\ul{\alpha},\ul{\varepsilon}} = 0$.

In $R=A[\mathbb{R}_{\geq 0}\times\mathbb{R}_{\geq 0}]$
the ideal $I_{(1,1),(0,1)}$  is graphed in Example~\ref{ex130508a}.
\end{ex}

We think of the  ideal $I_{\ul{\alpha},\ul{\varepsilon}}$ as ``almost principal'' since it is very close to the principal ideal $(\ul X^{\ul\alpha})R$.
Hence, a finite sum of ideals of the form $I_{\ul{\alpha},\ul{\varepsilon}}$ is ``almost finitely generated''.

\begin{prop}\label{pclassif2}
A monomial ideal $I \subseteq R$ has a finite m-irreducible decomposition if and only if it can be expressed as a finite sum of ideals of the form $\ideal{I}{\alpha}{\varepsilon}$.
 If $I$ has a finite m-irreducible decomposition,
then $I$ has a unique irredundant finite m-irreducible decomposition.
 \begin{proof}
  The first statement is proved as in \cite[Theorem 4.12]{ingebretson:dmirsr} using  \pref{classif1}.
  For existence in the second statement, let $I=\cap_{i=1}^mP_i$ be a finite m-irreducible decomposition.
If this decomposition is irredundant, then there is nothing to show. Assume that the decomposition is redundant,
so we have $P_k\subseteq P_j$ for some $k\neq j$. It follows that $I=\cap_{i\neq j}P_i$ is another  finite m-irreducible decomposition.
Continue removing redundancies in this way. The process terminates in finitely many steps since the original decomposition is finite.

To show uniqueness, let $I=\cap_{i=1}^mP_i=\cap_{j=1}^nQ_j$ be two irredundant finite m-irreducible decompositions.
For $s=1,\ldots,m$ it follows  that
$P_s\supseteq I=\cap_{j=1}^nQ_j$, so \lref{2.7} implies that
$$P_s=P_s+\bigcap_{j=1}^nQ_j=\bigcap_{j=1}^n(P_s+Q_j).$$
Since $P_s$ is m-irreducible, it follows that $P_s=P_s+Q_t\supset Q_t$ for some $t$.
Similarly, there is an index $u$ such that $Q_t\supseteq P_u$.
Hence, the irreduncancy of the intersection $\bigcap_{i=1}^mP_i$ implies that $P_s=P_u$, and thus $P_s=Q_t$.
We conclude that $\{P_1,\ldots,P_m\}\subseteq\{Q_1,\ldots,Q_n\}$.
Symmetrically, we have $\{P_1,\ldots,P_m\}\supseteq\{Q_1,\ldots,Q_n\}$,
so the decompositions are equal.
 \end{proof}
\end{prop}

\begin{cor}\label{cor130513a}
Every finitely generated monomial ideal $I \subseteq R$ has a finite m-irreducible decomposition.
\end{cor}

\begin{proof}
This follows from \eref{4.3} and \pref{classif2}.
\end{proof}

\begin{ex}\label{ex130508c}
Here we illustrate an m-irreducible decomposition of the ideal $I=I_{(1,1),(0,1)}$
in $R=A[\mathbb{R}_{\geq 0}\times\mathbb{R}_{\geq 0}]$
from Example~\ref{ex130508a}.
As with such decompositions in the standard polynomial ring $A[X_1,X_2]$,
the key is to use the graph of the monomial set $\monset I$ to find the decomposition.
The first diagram in the following display is the graph of $\monset{I}$.
The second one shows how we use the boundary lines from $\monset I$ to write $\monset I$ as the
intersection $\monset J\cap\monset K$ where $J$ and $K$ are generated by pure powers of $X_2$ and $X_1$, respectively.
The third and fourth diagrams show $J=J_{(\infty,1),(0,1)}$ and $K=J_{(1,\infty),(0,1)}$ separately.

\

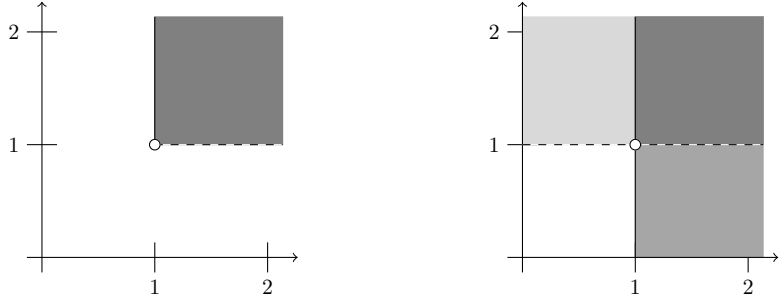
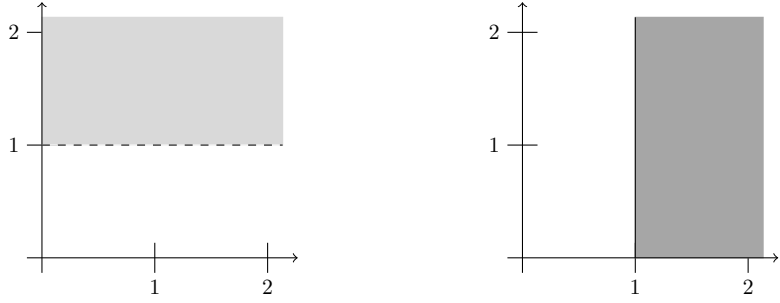
\begin{center}
\begin{tikzpicture}
	\draw[->] (-0.2,0) -- (3.4,0);
	\draw[->] (0,-0.2) -- (0,3.4);
	\draw (1.5,-0.2) node[below,scale=.75]{$1$} -- (1.5,0.2);
	\draw (3,-0.2) node[below,scale=.75]{$2$} -- (3,0.2);
	\draw (-0.2,1.5) node[left,scale=.75]{$1$} -- (0.2,1.5);
	\draw (-0.2,3) node[left,scale=.75]{$2$} -- (0.2,3);
	\draw[fill,color=black!50] (1.5,1.5) -- (3.2,1.5) -- (3.2,3.2) -- (1.5,3.2) -- cycle;
	\draw (1.5,1.5) -- (1.5,3.2);
	\draw[color=white] (1.5,1.5) -- (3.2,1.5);
	\draw[dashed] (1.5,1.5) -- (3.2,1.5);
	\draw (1.5,1.5)[fill,color=white] circle (2pt);
	\draw (1.5,1.5) circle (2pt);
\end{tikzpicture}
\qquad
\qquad
\qquad
\begin{tikzpicture}
	\draw (1.5,-0.2) node[below,scale=.75]{$1$} -- (1.5,0.2);
	\draw (3,-0.2) node[below,scale=.75]{$2$} -- (3,0.2);
	\draw (-0.2,1.5) node[left,scale=.75]{$1$} -- (0.2,1.5);
	\draw (-0.2,3) node[left,scale=.75]{$2$} -- (0.2,3);
	\draw[fill,color=black!15] (0,1.5) -- (3.2,1.5) -- (3.2,3.2) -- (0,3.2) -- cycle;
	\draw[fill,color=black!35] (1.5,0) -- (3.2,0) -- (3.2,3.2) -- (1.5,3.2) -- cycle;
	\draw[fill,color=black!50] (1.5,1.5) -- (3.2,1.5) -- (3.2,3.2) -- (1.5,3.2) -- cycle;
	\draw[->] (-0.2,0) -- (3.4,0);
	\draw[->] (0,-0.2) -- (0,3.4);
	\draw (1.5,0) -- (1.5,3.2);
	\draw[color=white] (0,1.5) -- (3.2,1.5);
	\draw[dashed] (0,1.5) -- (3.2,1.5);
	\draw (1.5,1.5)[fill,color=white] circle (2pt);
	\draw (1.5,1.5) circle (2pt);
\end{tikzpicture}
\end{center}

\

\begin{center}
\begin{tikzpicture}
	\draw (1.5,-0.2) node[below,scale=.75]{$1$} -- (1.5,0.2);
	\draw (3,-0.2) node[below,scale=.75]{$2$} -- (3,0.2);
	\draw (-0.2,1.5) node[left,scale=.75]{$1$} -- (0.2,1.5);
	\draw (-0.2,3) node[left,scale=.75]{$2$} -- (0.2,3);
	\draw[fill,color=black!15] (0,1.5) -- (3.2,1.5) -- (3.2,3.2) -- (0,3.2) -- cycle;
	\draw[->] (-0.2,0) -- (3.4,0);
	\draw[->] (0,-0.2) -- (0,3.4);
	\draw[color=white] (0,1.5) -- (3.2,1.5);
	\draw[dashed] (0,1.5) -- (3.2,1.5);
\end{tikzpicture}
\qquad
\qquad
\qquad
\begin{tikzpicture}
	\draw (1.5,-0.2) node[below,scale=.75]{$1$} -- (1.5,0.2);
	\draw (3,-0.2) node[below,scale=.75]{$2$} -- (3,0.2);
	\draw (-0.2,1.5) node[left,scale=.75]{$1$} -- (0.2,1.5);
	\draw (-0.2,3) node[left,scale=.75]{$2$} -- (0.2,3);
	\draw[fill,color=black!35] (1.5,0) -- (3.2,0) -- (3.2,3.2) -- (1.5,3.2) -- cycle;
	\draw[->] (-0.2,0) -- (3.4,0);
	\draw[->] (0,-0.2) -- (0,3.4);
	\draw (1.5,0) -- (1.5,3.2);
\end{tikzpicture}
\end{center}
One checks readily that $I=J\cap K$, say, using \pref{2.5}.

Note that Example~\ref{ex130508i} provides another m-irreducible decomposition. 
Moreover, it shows that one needs to be careful when using these
diagrams to generate decompositions, as the rough shape of the diagram (ignoring the distinction between 
dashed and solid lines, etc.) does not contain enough information.
\end{ex}

We conclude this section with a discussion of (possibly infinite) irredundant m-irreducible decompositions, beginning with existence.
\newcommand{\lola}[1]{\mathfrak{C}_{#1}}

\begin{prop}\label{p4.14}
Let $I$ be a monomial ideal in $R$, and let $\lola{I}$ denote the set of m-irreducible monomial ideals of $R$ that contain $I$.
Let $\lola{I}'$ denote the set of minimal elements of $\lola{I}$ with respect to containment.
\begin{enumerate}[\rm(a)]
\item \label{p4.14a}
For every $J\in \lola{I}$, there is an ideal $J'\in\lola{I}'$ such that $J'\subseteq J$.
\item \label{p4.14b}
With $\ul{1} = (1,\ldots,1)$, we have the following m-irreducible decompositions
$$I = \bigcap_{\mono{X}{r} \not \in I} J_{\ul{r},\ul{1}} =\bigcap_{J\in\lola{I}} J=\bigcap_{J\in\lola{I}'} J.$$
The third decomposition is irredundant.
\end{enumerate}
\end{prop}

\begin{proof}
\eqref{p4.14a}
Let $J\in\lola I$, and let $\lola{J,I}$ denote the set of ideals $K\in\lola{I}$ contained in $J$.
By Zorn's Lemma, it suffices to show that every chain $\mathfrak T$ in $\lola{J,I}$ has a lower bound in $\lola{J,I}$.
To this end, it suffices to show that the ideal $L:=\cap_{K\in\mathfrak T}K$ is m-irreducible.
If the chain $\mathfrak T$ has a minimal element, then $L$ is the minimal element, hence it is m-irreducible.
Thus, we assume that $\mathfrak T$ does not have a minimal element.
Fact~\ref{f2.2}\eqref{item130521d} shows that $L$ is a monomial ideal of $R$,
and the containments $L\subseteq J\subsetneq R$ implies that $L\neq R$.
Thus, to show that $L$ is m-irreducible, let $M$ and $N$ be monomial ideals of $R$ such that
$L=M\cap N$; we need to show that $L=M$ or $L=N$.
Since we have $L=M\cap N\subseteq M$, and similarly $L\subseteq N$,
it suffices to show that $L\supseteq M$ or $L\supseteq N$.

For each  $K\in\mathfrak T$, we have
$K\supset L=M\cap N$, so Lemma~\ref{l2.7} implies that
$$K=K+L=K+(M\cap N)=(K+M)\cap(K+N).$$
The fact that $K$ is m-irreducible implies that
either $K=K+M\supseteq M$ or $K=K+N\supseteq N$.

Case 1. For every $K\in\mathfrak T$, there is a $K'\in\mathfrak T$ such that $K\supseteq K'\supseteq M$.
In this case, it follows that $L=\cap_{K\in\mathfrak T}K\supseteq M$, as desired.

Case 2. There is an ideal $K\in\mathfrak T$ such that for every 
$K'\in\mathfrak T$ with $K\supseteq K'$, one has $K'\not\supseteq M$.
(Note that the fact that $\mathfrak T$ does not have a minimal element implies that 
at least one such $K'$ exists.)
From the paragraph before Case 1, we conclude that for every 
$K'\in\mathfrak T$ with $K\supseteq K'$, one has $K'\supseteq N$.
It now follows that $L\supseteq N$, as desired.

\eqref{p4.14b}
If $I=R$, then the desired conclusions are trivial since the empty intersection of ideals of $R$ is itself $R$.
Thus, we assume that $I\neq R$.
The equality $I = \bigcap_{\mono{X}{r} \not \in I} J_{\ul{r},\ul{1}}$
is proved like \cite[Proposition~4.14]{ingebretson:dmirsr}.
For each monomial $\mono{X}{r} \not \in I$, it follows that $J_{\ul{r},\ul{1}}\in\lola{I}$, so we have
the first containment in the following display.
\begin{align*}
I 
&= \bigcap_{\mono{X}{r} \not \in I} J_{\ul{r},\ul{1}}
\supseteq \bigcap_{J\in\lola{I}} J
\supseteq \bigcap_{J\in\lola{I}'} J
\supseteq I
\end{align*}
The second containment follows from part~\eqref{p4.14a},
and the third containment follows from the definition of $\lola{I}'$.
This establishes the desired decompositions.
Finally, the decomposition $\bigcap_{J\in\lola{I}'} J$ is irredundant because
there are no proper containments between minimal elements of$\lola I$, by definition.
\end{proof}

The following example shows that infinite irredundant m-irreducible decompositions need not be unique.

\begin{ex}\label{ex131224a}
Set $d=2$ and $ I= \left( \{ X^r Y^{1-r} \mid 0 \leq r \leq 1 \} \right)R$ with $\mathbb M_i=\mathbb R$ for $i=1,2$.
In~\cite[Example 4.13]{ingebretson:dmirsr}, it is shown that $I$ does not admit a finite m-irreducible decomposition.
However, it is straightforward to show that the following m-irreducible decompositions are irredundant and distinct:
$$I=\bigcap_{0\leq r\leq 1}J_{(r,1-r),(1,1)}=\bigcap_{0\leq r\leq 1}J_{(r,1-r),(1,0)}.$$
Moreover, one can use this idea to construct infinitely many distinct irredundant m-irreducible decompositions of $I$.
Indeed, for every subset $S$ of the closed interval $[0,1]$, we have
$$I=\left(\bigcap_{r\in S}J_{(r,1-r),(1,1)}\right)\bigcap\left(\bigcap_{r\in[0,1]\smallsetminus S}J_{(r,1-r),(1,0)}\right).$$

\end{ex}

\section{An Extended Example} \label{sec130510a}

\newcommand{\sss}{s}
\newcommand{\SSS}{S}
\newcommand{\ttt}{\varepsilon}

Here we show how to use discrete techniques
from~\cite{paulsen:eiwg} to compute some decompositions in our setting.
This section's main result is \tref{hm130513a}.

\begin{Assumption}\label{a130510a}
Throughout this section,  $I$ is an ideal in the ring
$R = A[\mathbb{R}_{\geq 0}\times\cdots\times\mathbb{R}_{\geq 0}]$
generated by a non-empty set of monomials of the form
$X_i^aX_j^a$ with $i\neq j$ and $a\in \mathbb{R}_{> 0}$. 
Also, we consider the standard polynomial ring $S=A[X_1,\ldots,X_d]$.
Let $\Omega$ denote the following set of intervals:
$$\Omega=\{(a,\infty)\mid a\in\mathbb R_{\geq 0}\}\cup \{[b,\infty)\mid b\in\mathbb R_{> 0}\}.$$
\end{Assumption}

\begin{ex}\label{ex130508d}
In the case $d=3$, we may consider the ideal
$$I=(X_1^aX_2^a,X_2^bX_3^b\mid a\geq 1,b>2)R=(X_1X_2,X_2^bX_3^b\mid b>2)R.$$
By \pref{2.5} it is routine to verify the  irredundant decomposition
\begin{align*}
I&=(X_1,X_2^b\mid b>2)R\cap(X_1,X_3^b\mid b>2)R\cap(X_2)R\\
&=J_{(1,\infty,2),(0,0,1)}\cap J_{(\infty,1,\infty),(0,0,0)}\cap J_{(1,2,\infty),(0,1,0)}.
\end{align*}
\end{ex}

\begin{notation}\label{notn130512a}
Define $\Gamma$ to be the finite simple graph with vertex set $V=\{1,\ldots,d\}$ and edge set 
$$E=\{ij\mid \text{$i\neq j$ and $X_i^aX_j^a\in I$ for some $a>0$}\}$$
where $ij=\{i,j\}$.
For each $ij\in E$, set
\begin{align*}
\SSS(ij)&=\{a>0\mid X_i^aX_j^a\in I\}. 
\\
\sss(ij)&=\inf \SSS(ij)\\
\ttt(ij)&=\begin{cases} 0 & \text{if $\sss(ij)\in \SSS(ij)$} \\ 1 & \text{if $\sss(ij)\notin \SSS(ij)$.}\end{cases}
\end{align*}
This defines  functions 
$\sss\colon E\to\mathbb R_{\geq 0}$ and $\ttt\colon E\to\{0,1\}$.
\end{notation}

\begin{ex}\label{ex130510a}
Continue with the ideal $I$ from Example~\ref{ex130508d}.
The graph $\Gamma$ in this case is
$$\xymatrix{1\ar@{-}[r] &2\ar@{-}[r]&3.}$$
And the values $\SSS(ij)$
are
\begin{align*}
\SSS(12)&=[1,\infty)&
\SSS(23)&=(2,\infty)
\\
\sss(12)&=1 & \sss(23)&=2 \\
\ttt(12)&=0&\ttt(23)&=1.
\end{align*}
\end{ex}

\begin{fact}\label{fact130510a}
For each $ij\in E$, the set $\SSS(ij)$ is an interval. Indeed, if $a\in \SSS(ij)$, then $X_i^aX_j^a\in I$, so for all $r>0$, we have
$X_i^{r+a}X_j^{r+a}=X_i^{r}X_j^{r}X_i^{a}X_j^{a}\in I$, implying that $r+a\in \SSS(ij)$.
Moreover, it is straightforward to show that
$$\SSS(ij)=\begin{cases} [\sss(ij),\infty) & \text{if $\ttt(ij)=0$} \\ (\sss(ij),\infty) & \text{if $\ttt(ij)=1$.}\end{cases}
$$
In particular, we have a function $\SSS\colon E\to\Omega$.

The ideal $I$ is a finite sum $\sum_{ij\in E}I_{\underline\alpha(ij),\underline\epsilon(ij)}$. 
This essentially follows from the previous paragraph, with the following definitions of $\underline\alpha(ij)$ and $\underline\epsilon(ij)$:
\begin{align*}
\alpha(ij)_k&=\begin{cases}
\infty & \text{if $k\notin\{i,j\}$} \\
\sss(ij) & \text{if $k\in\{i,j\}$}\end{cases}
&
\epsilon(ij)_k&=\begin{cases}
1 & \text{if $k\notin\{i,j\}$} \\
\ttt(ij) & \text{if $k\in\{i,j\}$.}\end{cases}
\end{align*}
\pref{classif2} implies that $I$ has a finite m-irreducible decomposition. 
\end{fact}

\begin{ex}\label{ex130510e}
With the ideal $I$ from Example~\ref{ex130508d},
we have
$$I=I_{(1,1,\infty),(0,0,1)}+I_{(\infty,2,2),(1,1,1)}.$$
\end{ex}

\begin{defin}\label{defn130510a}
A \textit{vertex cover} of the graph $\Gamma$ is a subset $W\subseteq V$ such that
for all $ij\in E$ we have either $i\in W$ or $j\in W$.
A vertex cover is \textit{minimal} if it is minimal in the set of all vertex covers with respect to inclusion.
\end{defin}

\begin{ex}\label{ex130510b}
Continue with the ideal $I$ from Example~\ref{ex130508d}.
The graph $\Gamma$ in this case 
has two minimal vertex covers, namely $\{1,3\}$ and $\{2\}$. It has the following non-minimal vertex covers: $\{1,2\}$, $\{2,3\}$, and $\{1,2,3\}$.
We will see below that the  irredundant m-irreducible decomposition of $I$ is given by the vertex covers
$\{1,2\}$, $\{1,3\}$, and $\{2\}$ with some additional data. 
\end{ex}

The work from~\cite{paulsen:eiwg} takes its cue from the following decomposition result that we know from~\cite[Proposition 6.1.16]{villarreal:ma}.

\begin{fact}\label{fact130610a}
The \textit{edge ideal} of the finite simple graph $\Gamma$ is the ideal $I(\Gamma)=(X_iX_j\mid ij\in E)S$.
Then we have the following m-irreducible decompositions
$$I(\Gamma)=\bigcap_{W}Q_W=\bigcap_{\text{$W$ min}}Q_W$$
where the first intersection is taken over all vertex covers of $\Gamma$, and the second
intersection is taken over all minimal vertex covers of $\Gamma$. The second intersection is irredundant.
\end{fact}

\begin{ex}\label{ex130510d}
Continue with the graph $\Gamma$ from Example~\ref{ex130510a}.
Using the minimal vertex covers from Example~\ref{ex130510b}, we have
$$I(\Gamma)=Q_{\{1,3\}}\cap Q_{\{2\}}=(X_1,X_3)S\cap(X_2)S.$$
One can verify these equalities using \pref{2.5}, and the irredundancy is straightforward.
\end{ex}

To prepare for the decomposition result for the ideal $I$, we review the decomposition result from~\cite{paulsen:eiwg}
for weighted edge ideals.

\begin{defin}\label{defn130510b}
A \textit{weight function} for the  graph $\Gamma$ is a function $\omega\colon E\to\mathbb Z_{>0}$.
For each  $ij\in E$, the value $\omega(ij)$ is the \textit{weight} of the edge $ij$.
Write $\Gamma_{\omega}$ for the ordered pair $(\Gamma,\omega)$.
Fix a weight function $\omega$ of $\Gamma$.

A \textit{weighted vertex cover} of $\Gamma_{\omega}$ is
a pair $(W,\delta)$ such that
\begin{enumerate}[(1)]
\item $W$ is a vertex cover of $\Gamma$, and
\item $\delta\colon W\to \mathbb Z_{>0}$ is a function such that for each edge $ij\in E$, either
\begin{enumerate}[(a)]
\item $i\in W$ and $\delta(i)\leq\omega(ij)$, or
\item $j\in W$ and $\delta(j)\leq\omega(ij)$.
\end{enumerate}
\end{enumerate}
The value $\delta(i)$ is the \textit{weight} of the vertex $i$.

Given two weighted vertex covers 
$(W,\delta)$ and $(W',\delta')$
of $\Gamma_{\omega}$, we write
$(W,\delta)\leq(W',\delta')$
provided that
\begin{enumerate}[(1)]
\item $W\subseteq W'$, and
\item for all $i\in W$, we have $\delta(i)\geq\delta'(i)$.
\end{enumerate}
A weighted vertex cover is \textit{minimal} if it is minimal in the set of all weighted vertex covers with respect to this ordering.

Given a weighted vertex cover $(W,\delta)$ 
of $\Gamma_{\omega}$, set
$$P_{W,\delta}=(X_i^{\delta(i)}\mid i\in W)S.$$
The \textit{weighted edge ideal} of $\Gamma_{\omega}$ is the ideal
$$I(\Gamma_{\omega})=(X_i^{\omega(ij)}X_i^{\omega(ij)}\mid ij\in E)S.$$
\end{defin}

\begin{fact}\label{fact130510b}
Given a weight function $\omega$ of $\Gamma$,
we have the following m-irreducible decompositions
$$I(\Gamma_{\omega})=\bigcap_{(W,\delta)}P_{W,\delta}=\bigcap_{\text{$(W,\delta)$ min}}P_{W,\delta}$$
where the first intersection is taken over all weighted vertex covers of $\Gamma_{\omega}$, and the second
intersection is taken over all minimal weighted vertex covers of $\Gamma_{\omega}$. The second intersection is irredundant.
See~\cite[Theorem 3.5]{paulsen:eiwg}.

This technique allows us to find an irredundant m-ireducible decomposition for any ideal $J$ in $S$ generated by monomials of the form
$X_i^aX_j^a$ with $i\neq j$ and $a\in\mathbb Z_{>0}$. 
Indeed, such an ideal is of the form $I(\Gamma_\omega)$ where $ij$ is an edge of $\Gamma$ if and only if the monomial $X_i^aX_j^a$ is in $J$
for some $a\in\mathbb Z_{>0}$,
and $\omega(ij)$ is the least $a$ such that $X_i^aX_j^a\in J$.
\end{fact}

\begin{ex}\label{ex130510c}
Continue with the graph $\Gamma$ from Example~\ref{ex130510a}.
Consider the weight function $\omega$ with
$\omega(12)=1$ and $\omega(23)=2$. We represent this graphically
by labeling each edge $ij$ with the value $\omega(ij)$:
$$\xymatrix{1\ar@{-}[r]^-1 &2\ar@{-}[r]^-2&3.}$$
We also represent weighted vertex covers graphically with a box around each vertex in the vertex cover 
and using a superscript for the weight, as follows:
$$\xymatrix{*+[F]{1^1}\ar@{-}[r]^-1 &2\ar@{-}[r]^-2&*+[F]{3^2}.}$$
This weighted graph has three minimal weighted vertex covers, the one represented above, and the next two:
$$\xymatrix{1\ar@{-}[r]^-1 &*+[F]{2^1}\ar@{-}[r]^-2&3\\
*+[F]{1^1}\ar@{-}[r]^-1 &*+[F]{2^2}\ar@{-}[r]^-2&3.}$$
Note that the first two correspond to minimal vertex covers of the unweighted graph $\Gamma$, but the third one does not.
The  irredundant decomposition of $I(\Gamma_{\omega})$ coming from \fref{act130510b} is 
$$I(\Gamma_{\omega})=(X_1,X_2^2)S\cap(X_1,X_3^2)S\cap(X_2)S.$$
One can check this equality using \pref{2.5}, and the irredundancy is straightforward.
\end{ex}

Now we develop a version of this construction for the  ideal $I$ from Assumption~\ref{a130510a}.

\begin{defin}\label{defn130510c}
Let $\Gamma_{\SSS}$ denote the ordered pair $(\Gamma,\SSS)$
where $S$ is from Notation~\ref{notn130512a}.
An \textit{interval vertex cover} of $\Gamma_{\SSS}$ is
a pair $(W,\sigma)$ such that
\begin{enumerate}[(1)]
\item $W$ is a vertex cover of $\Gamma$, and
\item $\sigma\colon W\to \Omega$ is a function such that for each edge $ij\in E$, either
\begin{enumerate}[(a)]
\item $i\in W$ and $\SSS(ij)\subseteq\sigma(i)$, or
\item $j\in W$ and $\SSS(ij)\subseteq\sigma(j)$.
\end{enumerate}
\end{enumerate}
The value $\sigma(i)$ is the ``interval weight'' of the vertex $i$.

Given  ordered pairs 
$(W,\sigma)$ and $(W',\sigma')$
where $W,W'\subseteq V$ are subsets and $\sigma\colon W\to\Omega$ and $\sigma'\colon W'\to\Omega$ are functions,  write
$(W,\sigma)\leq(W',\sigma')$
whenever
\begin{enumerate}[(1)]
\item $W\subseteq W'$, and
\item for all $i\in W$, we have $\sigma(i)\subseteq\sigma'(i)$.
\end{enumerate}
An interval vertex cover of $\Gamma_{\SSS}$ is \textit{minimal} if it is minimal in the set of all interval vertex covers with respect to this ordering.

Given an ordered pair $(W,\sigma)$ 
where $W\subseteq V$ and $\sigma\colon W\to\Omega$, set
$$Q_{W,\sigma}=(X_i^{a}\mid \text{$i\in W$ and $a\in\sigma(i)$})R.$$
\end{defin}

\begin{ex}\label{ex130511a}
Continue with the ideal $I$ from Example~\ref{ex130508d}.
We visualize the associated data from Example~\ref{ex130510a}
similarly to the labeled graph from Example~\ref{ex130510c},  keeping track of the entire interval $\SSS(ij)$
for each edge $ij$:
$$\xymatrix{1\ar@{-}[r]^-{\geq 1} &2\ar@{-}[r]^-{>2}&3.}$$
We also represent interval vertex covers graphically with a box around each vertex in the vertex cover 
and  a superscript for the interval weight, as follows:
$$\xymatrix{(W_1,\sigma_1):&*+[F]{1^{\geq 1}}\ar@{-}[r]^-{\geq 1} &2\ar@{-}[r]^-{>2}&*+[F]{3^{>2}}.}$$
This weighted graph has three minimal interval vertex covers, the one represented above and the next two:
$$\xymatrix{(W_2,\sigma_2):&1\ar@{-}[r]^-{\geq 1} &*+[F]{2^{\geq 1}}\ar@{-}[r]^-{>2}&3\\
(W_3,\sigma_3):&*+[F]{1^{\geq 1}}\ar@{-}[r]^-{\geq 1} &*+[F]{2^{>2}}\ar@{-}[r]^-{>2}&3.}$$
Again, the first two correspond to minimal vertex covers of the unweighted graph $\Gamma$, but the third one does not.
\end{ex}

\begin{fact}\label{fact130510d}
The ideals in $R$ of the form $Q_{W,\sigma}$ are exactly the ideals of the form $J_{\underline\alpha,\underline\epsilon}$,
since they are exactly the ones generated by (intervals of) pure powers of the variables.
Thus, the ideals of the form $Q_{W,\sigma}$ are exactly the m-irreducible ideals of $R$ by \pref{classif1}.
\end{fact}

The next lemma is the key to the main result of this section.

\begin{lem}\label{lem130511a}
Consider  ordered pairs 
$(W,\sigma)$ and $(W',\sigma')$
where $W,W'\subseteq V$ are subsets and $\sigma\colon W\to\Omega$ and $\sigma'\colon W'\to\Omega$ are functions.
\begin{enumerate}[\rm(a)]
\item\label{lem130511a1} 
One has $Q_{W,\sigma}\subseteq Q_{W',\sigma'}$
if and only if $(W,\sigma)\leq(W',\sigma')$.
\item\label{lem130511a2} 
One has $I\subseteq Q_{W,\sigma}$  if and only if $(W,\sigma)$ is an interval vertex cover of~$\Gamma_{\SSS}$.
\end{enumerate}
\end{lem}

\begin{proof}
\eqref{lem130511a1}
We prove the forward implication; the  reverse implication is  similar and easier.
Assume that $Q_{W,\sigma}\subseteq Q_{W',\sigma'}$.
To show that $(W,\sigma)\leq(W',\sigma')$, let $i\in W$;
we need to show that $i\in W'$ and that $\sigma(i)\subseteq\sigma'(i)$.
Let $a\in\sigma(i)$.
By definition, it follows that $X_i^a\in Q_{W,\sigma}\subseteq Q_{W',\sigma'}$.
\fref{2.3} implies that $X_i^a$ is a multiple of a monomial generator of $Q_{W',\sigma'}$,
so there exist $j\in W'$ and $b\in\sigma'(j)$ such that
$X_i^a\in(X_j^b)R$. Note that $a,b>0$. It follows that
$i=j\in W'$ and $b\leq a$. Since $\sigma'(i)$ is an interval of the form $[c,\infty)$ or $(c,\infty)$,
the conditions $b\in\sigma'(i)$ and $b\leq a$ implies that $a\in\sigma'(i)$, as desired.

\eqref{lem130511a2}
Again, we prove the forward implication.
Assume that $I\subseteq Q_{W,\sigma}$.
To show that $W$ is a vertex cover of $\Gamma$, let $ij\in E$. 
By definition of $E$ and $\SSS(ij)$, there is an element $a\in\SSS(ij)$ such that
$X_i^aX_j^a\in I\subseteq Q_{W,\sigma}$. By \fref{2.3},
the element $X_i^aX_j^a$ is a multiple of a monomial generator of $Q_{W,\sigma}$,
so there exist $k\in W$ and $b\in\sigma(k)$ such that $X_i^aX_j^a\in (X_k^b)R$.
Since $a,b>0$, we conclude that either $i=k\in W$ or $j=k\in W$, so $W$ is a vertex cover of $\Gamma$.

To show that $(W,\sigma)$ is an interval vertex cover of $\Gamma$, let $ij\in E$. 
We proceed by cases.

Case 1: $i\notin W$.
Since $W$ is a vertex cover, we have $j\in W$.
In this case, we need to show that $\SSS(ij)\subseteq\sigma(j)$, so let $a\in\SSS(ij)$.
It follows that $X_i^aX_j^a\in I\subseteq Q_{W,\sigma}$.
Hence, there is a monomial generator $X_k^b\in Q_{W,\sigma}$ such that
$X_i^aX_j^a\in(X_k^b)R$. Since $i\notin W$, the first paragraph of the proof of part~\eqref{lem130511a2}
shows that $j=k$, and we conclude that $b\leq a$. 
As in the proof of part~\eqref{lem130511a1}, we conclude that $a\in\sigma(j)$, as desired.

Case 2: $j\notin W$. This case is handled like Case 1.

Case 3: $i,j\in W$ and $\SSS(ij)\not\subseteq\sigma(i)$.
Again, we need to show that $\SSS(ij)\subseteq\sigma(j)$, so let $a\in\SSS(ij)$.
The condition $\SSS(ij)\not\subseteq\sigma(i)$ implies that there is an element $a'\in\SSS(ij)\smallsetminus\sigma(i)$.
If $a'\leq a$, then it suffices to show that $a'\in\sigma(j)$, as above.
If $a\leq a'$, then the assumption $a'\in\SSS(ij)\smallsetminus\sigma(i)$
implies that $a\in\SSS(ij)\smallsetminus\sigma(i)$ because of
the shape of the interval $\sigma(i)$.
Thus, we may replace $a$ by $a'$ if necessary to assume that $a\in\SSS(ij)\smallsetminus\sigma(i)$.

As above,
there is a monomial generator $X_k^b\in Q_{W,\sigma}$ such that
$X_i^aX_j^a\in(X_k^b)R$.
It follows that either $i=k$ or $j=k$, and $b\leq a$.
If $i=k$, then the condition $X_k^b\in Q_{W,\sigma}$ implies that $b\in\sigma(i)$;
hence the inequality $b\leq a$ implies that $a\in\sigma(i)$ because of the shape of $\sigma(i)$;
this is a contradiction.
Thus, we have $j=k$ and, as in the previous sentence, $a\in\sigma(j)$.
\end{proof}

\begin{thm}\label{thm130513a}
For the ideal $I$ from Assumption~\ref{a130510a},
we have the following m-irreducible decompositions
$$I=\bigcap_{(W,\sigma)}Q_{W,\sigma}=\bigcap_{\text{$(W,\sigma)$ min}}Q_{W,\sigma}$$
where the first intersection is taken over all interval vertex covers of $\Gamma$, and the second
intersection is taken over all minimal interval vertex covers of $\Gamma$. The second intersection is finite and irredundant,
and
the graph $\Gamma$ with data from Notation~\ref{notn130512a}
has only finitely many minimal interval vertex covers.
\end{thm}

\begin{proof}
In the next display, the  containment is from \lref{em130511a}\eqref{lem130511a2}.
$$I\subseteq\bigcap_{(W,\sigma)}Q_{W,\sigma}=\bigcap_{\text{$(W,\sigma)$ min}}Q_{W,\sigma}$$
For the equality, we have $\bigcap_{(W,\sigma)}Q_{W,\sigma}\subseteq\bigcap_{\text{$(W,\sigma)$ min}}Q_{W,\sigma}$
by basic properties of intersections, and the reverse containment follows from \lref{em130511a}\eqref{lem130511a1}.
Also, the second intersection is irredundant by \lref{em130511a}\eqref{lem130511a1}.

By \pref{classif2} and \fref{act130510a}, the ideal $I$ has a finite m-irreducible decomposition
which is of the form
$I=\cap_{k=1}^mQ_{W_k,\sigma_k}$
by \fref{act130510d}. Note that \lref{em130511a}\eqref{lem130511a2} implies that each pair $(W_k,\sigma_k)$ is an
interval vertex cover of $\Gamma$.
Thus, we have
$$I=\bigcap_{k=1}^mQ_{W_k,\sigma_k}\supseteq \bigcap_{(W,\sigma)}Q_{W,\sigma}.$$
With the previous display, this provides the equalities from the statement of the result.
Thus, it remains to show that the intersection $\bigcap_{\text{$(W,\sigma)$ min}}Q_{W,\sigma}$ is finite.
For this, let $(W,\sigma)$ be a minimal interval vertex cover of $\Gamma$. It suffices to show that 
$(W,\sigma)=(W_k,\sigma_k)$ for some $k$.
From the equalities we have already established, we have
$$Q_{W,\sigma}\supseteq I=\bigcap_{k=1}^mQ_{W_k,\sigma_k}.$$
The proof of \pref{classif2} shows that
$Q_{W,\sigma}\supseteq Q_{W_k,\sigma_k}$
for some $k$. So, \lref{em130511a}\eqref{lem130511a1} implies that $(W,\sigma)\geq(W_k,\sigma_k)$.
As these are both interval vertex covers of $\Gamma$, the minimality of $(W,\sigma)$
yields the desired equality $(W,\sigma)=(W_k,\sigma_k)$.
\end{proof}

\begin{ex}\label{ex130511b}
Continue with the ideal $I$ from Example~\ref{ex130508d}.
The minimal vertex covers from Example~\ref{ex130511b} provide the following irredundant m-irreducible decomposition
$$I=Q_{(W_1,\sigma_1)}\cap Q_{(W_1,\sigma_2)}\cap Q_{(W_3,\sigma_3)}$$
which is exactly the decomposition computed in Example~\ref{ex130508d}.
\end{ex}

\section{Monomial Krull Dimension}\label{smdim}
We now introduce and study a notion of  Krull dimension for this setting. 
The main result of this section is  \cref{or130521a}.

\begin{Assumption}
Throughout this section, set
$R = A[\moncrs]$. 
\end{Assumption}

\begin{defin}\label{dmdim}
 For an m-prime ideal $P=Q_T$, we employ the notation $v(P)$ to denote the number of variables from which $P$ is generated, i.e.,
 \[ v(P) =|T|= |\{ X_i \st X_i^k \in P \text{ for some } k \in \mathbb{M}^i_{> 0} \text{ and } 1 \leq i \leq d \}|\]
and set $\mdimen^*(R/P) =  d - v(P)$.

The \textit{monomial Krull dimension} of an arbitrary monomial ideal $I$ is
 \[\mdim{R/I} = \sup\{\mdimen^*(R/Q) \st Q \text{ is m-prime and } Q \supseteq I\}.\]
\end{defin}

\begin{fact}\label{fsubdim}
By definition, if $I$ and $J$ are monomial ideals in $R$ such that $I \supseteq J,$ then $\mdim{R/I} \leq \mdim{R/J}$, and
furthermore we have the following.
 \[\mdim{R/I} = \begin{cases}
 \max\{\mdimen^*(R/Q) \st Q \text{ is m-prime and } Q \supseteq I\} &\text{if $I\neq R$}\\
 -\infty&\text{if $I= R$}\end{cases}
 \]
Also, given an m-prime ideal $P$ of $R$, it is straightforward to show that $\mdim{R/P}=\mdimen^*(R/P)$.
\end{fact}

\begin{ex}\label{ex130509a}
Let $I$ be a monomial ideal in the ring $R$.
It is straightforward to show that $\mdim{R/I}=d$ if and only if $I=0$.
Also, if $I\neq R$, then $\mdim{R/I}=0$ if and only if for $i=1,\ldots,d$ there is 
an element $a_i\in\mathbb M^i_{>0}$ such that $X_i^{a_i}\in I$.

In the case $d=2$, this tells us that the ideals from Examples~\ref{ex130508a}
and~\ref{ex130508c} have $\mdim{R/I}=1$, and the ideals from Examples~\ref{ex130508g} and~\ref{ex130508i} have $\mdim{R/I}=0$.
\end{ex}

Before proving  our main results, we verify some desired properties m-dim.

\begin{prop}\label{pdimchain}
 For a monomial ideal $I$ in $R$, we have
 \[ \mdim{R/I} = \sup\{n \geq 0 \; | \; \exists \text{ m-prime } P_0 \subset P_1 \subset \cdots \subset P_n \text{ and } I \subseteq P_0 \}. \]
 \begin{proof}
Assume without loss of generality that $I\neq R$.

  Let $m=\max\{n \geq 0 \; | \; \exists \text{ m-prime } P_0 \subset P_1 \subset \cdots \subset P_n \text{ and } I \subseteq P_0 \}$ with corresponding maximal chain $I \subseteq P_0 \subset P_1 \subset \cdots \subset P_m$. 
  The maximality of this chain implies that $P_m=Q_{\{1,\ldots,d\}}$.
  Then for $i=0,\ldots,m-1$ we have $v(P_{i+1})=v(P_i)+1$; otherwise, we could find an m-prime $P$ such that $P_i \subset P \subset P_{i+1},$ which would contradict the maximality of $m$. Hence, we have $d=v(P_{m})=v(P_0)+m$, and it follows that $m = \mdim{R/P_0}$. It remains to show  that $\mdim{R/P_0} = \mdim{R/I}$.
  
  We have that $\mdim{R/P_0} \leq \mdim{R/I}$ by \fref{subdim}. Now, suppose that $\mdim{R/P_0}<\mdim{R/I}$. That would mean there is an m-prime ideal $P \supseteq I$ 
  such that $v(P)<v(P_0)$. We could then create a chain $P \subset P_0' \subset \cdots \subset P_m'$,
  contradicting the maximality of $m$. 
 \end{proof}
\end{prop}

The next result applies whenever $I$ is finitely generated by \cref{or130513a}.

\begin{prop}
Given a monomial ideal $I$ in $R$ with a finite m-irreducible decomposition $I = \bigcap_{i=1}^t J_i$, one has
 \[ \mdim{R/I} = \sup_{i}\{\mdim{R/J_i}\}.\]
 \begin{proof}
Assume without loss of generality that $t\geq 1$, i.e., that $I\neq R$.
Since $J_i \supseteq I$ for all $i$,  \fref{subdim} 
implies that $\sup_{i}\{\mdim{R/J_i}\}\leq \mdim{R/I}$.

  We now claim that for every m-prime ideal $P,$ if $P \supseteq I,$ then there is an index $k$ such that $P \supseteq J_k$. By way of contradiction, suppose that for all $k$ there is a monomial $f_k \in \monset{J_k}\smallsetminus\monset{P}$. Since $P$  is m-prime and $f_k \notin \monset{P}$ for all $k,$ we have that  $\prod_{k=1}^t f_k \notin \monset{P}$. However, $f_k \in J_k$ implies that $\prod_{k=1}^tf_k \in \bigcap_{k=1}^tJ_k = I \subseteq P$, a contradiction. 

Now, let $P$ be  m-prime such that $I\subseteq P$ and $\mdim{R/P}=\mdim{R/I}$. The claim implies that there is an index $k$ such that $P \supseteq J_k$,
so we have 
$$\mdim{R/I}=\mdim{R/P}\leq\mdim{R/J_k}\leq\sup_{i}\{\mdim{R/J_i}\}$$
  as desired. 
 \end{proof}
\end{prop}

\begin{ex}
Consider the case $\mathbb{M}^i=\mathbb R$ for all $i$
and the ideal $I$ from Assumption~\ref{a130510a}.
Using ideas from~\cite{paulsen:eiwg}, one shows that $\mdim{R/I}=d-\tau(\Gamma)$
where $\tau(\Gamma)$ is the \textit{vertex cover number} of $\Gamma$
$$\tau(\Gamma)=\min\{|W|\mid\text{$W$ is a vertex cover of $\Gamma$}\}.$$
\end{ex}

To consider the semicontinuity of monomial Krull dimension, we introduce the next definition.

\begin{defin}\label{def130521a}
 Let $\varepsilon \in \mathbb{R}_{> 0}$. For monomial ideals $I,J$ in $R$, we say that dist($I,J$) $< \varepsilon$ if 
 \begin{enumerate}[(1)]
\item\label{def130521a1} for all $\underline{X}^{\underline{\gamma}} \in \monset{I},$ there exists $\underline{X}^{\underline{\delta}} \in \monset{J}$ such that dist($\underline{\gamma},\underline{\delta}$) $< \varepsilon,$ and 
\item\label{def130521a2} for all $\underline{X}^{\underline{\delta}'} \in \monset{J},$ there exists $\underline{X}^{\underline{\gamma}'} \in \monset{I}$ with $\dist{\underline{\delta}'}{\underline{\gamma}'} < \varepsilon$. 
\end{enumerate}
Here dist($\underline{\gamma},\underline{\delta}$) = $|\underline{\gamma}-\underline{\delta}| = \sqrt{(\gamma_1-\delta_1)^2+\cdots+(\gamma_d-\delta_d)^2}$.
\end{defin}

\begin{defin}\label{def130521b}
  The distance between two monomial ideals $I$, $J$ in $R$ is \[\dist{I}{J} = \inf\{\varepsilon > 0 \st \dist{I}{J} < \varepsilon \}.\]
\end{defin}

\begin{ex}\label{ex130508f}
Let $I$ be a  monomial ideal in  $R$.
Since $\monset 0=\emptyset$, we have
$$\dist I0=\begin{cases}
0 & \text{if $I=0$} \\
\infty &\text{if $I\neq 0$}\end{cases}
$$
\end{ex}

\begin{ex}
\label{ex130509c}
The ideals $J_{2,1,0}=(X_2)R$ and $J_{1,1,0}=(X_1)R$
in $R=A[\mathbb{R}_{\geq 0}\times\mathbb{R}_{\geq 0}]$
satisfy $\dist{(X_2)R}{(X_1)R}=1$.

\

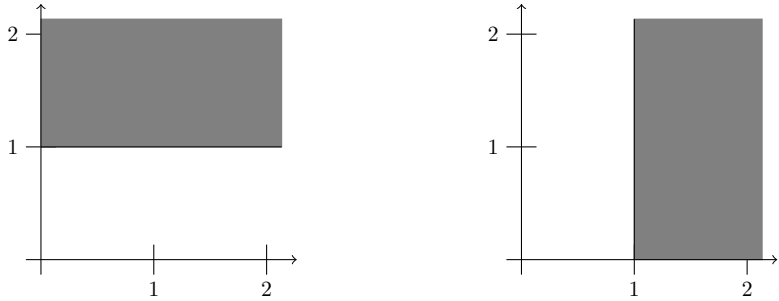
\begin{center}
\begin{tikzpicture}
	\draw (1.5,-0.2) node[below,scale=.75]{$1$} -- (1.5,0.2);
	\draw (3,-0.2) node[below,scale=.75]{$2$} -- (3,0.2);
	\draw (-0.2,1.5) node[left,scale=.75]{$1$} -- (0.2,1.5);
	\draw (-0.2,3) node[left,scale=.75]{$2$} -- (0.2,3);
	\draw[fill,color=black!50] (0,1.5) -- (3.2,1.5) -- (3.2,3.2) -- (0,3.2) -- cycle;
	\draw[->] (-0.2,0) -- (3.4,0);
	\draw[->] (0,-0.2) -- (0,3.4);
	\draw (0,1.5) -- (3.2,1.5);
\end{tikzpicture}
\qquad
\qquad
\qquad
\begin{tikzpicture}
	\draw (1.5,-0.2) node[below,scale=.75]{$1$} -- (1.5,0.2);
	\draw (3,-0.2) node[below,scale=.75]{$2$} -- (3,0.2);
	\draw (-0.2,1.5) node[left,scale=.75]{$1$} -- (0.2,1.5);
	\draw (-0.2,3) node[left,scale=.75]{$2$} -- (0.2,3);
	\draw[fill,color=black!50] (1.5,0) -- (3.2,0) -- (3.2,3.2) -- (1.5,3.2) -- cycle;
	\draw[->] (-0.2,0) -- (3.4,0);
	\draw[->] (0,-0.2) -- (0,3.4);
	\draw (1.5,0) -- (1.5,3.2);
\end{tikzpicture}
\end{center}
This is intuitively clear from the above diagrams.
To verify this rigorously, first note that, for every monomial $X_1^aX_2^b\in(X_2)R$ one has $X_1^{a+1}X_2^b\in(X_1)R$ and $\dist{X_1^aX_2^b}{X_1^{a+1}X_2^b}=1$.
Similarly, every monomial in $(X_1)R$ is distance 1 from a monomial in $(X_2)R$. This implies that $\dist{(X_2)R}{(X_1)R}<1+\epsilon$ for all $\epsilon>0$,
so $\dist{(X_2)R}{(X_1)R}\leq1$. Finally, note that the  monomial in $(X_2)R$ that is closest to $X_1\in(X_1)R$ is $X_1X_2$, which is distance 1 from $X_1$,
so $\dist{(X_2)R}{(X_1)R}\geq1$.
\end{ex}

\begin{ex}\label{ex130509b}
The ideals $I_{(1,1),(0,0)}$ and $I_{(1,1),(1,1)}$ 
in $R=A[\mathbb{R}_{\geq 0}\times\mathbb{R}_{\geq 0}]$

\

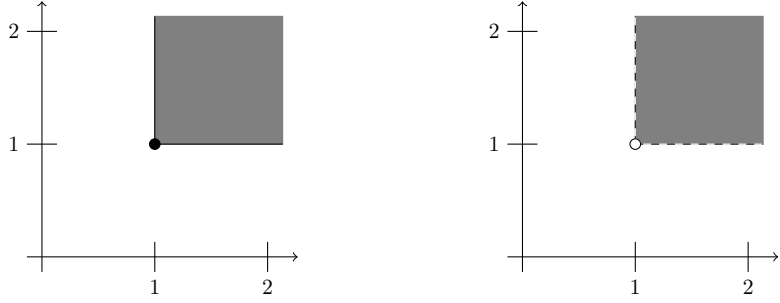
\begin{center}
\begin{tikzpicture}
	\draw[->] (-0.2,0) -- (3.4,0);
	\draw[->] (0,-0.2) -- (0,3.4);
	\draw (1.5,-0.2) node[below,scale=.75]{$1$} -- (1.5,0.2);
	\draw (3,-0.2) node[below,scale=.75]{$2$} -- (3,0.2);
	\draw (-0.2,1.5) node[left,scale=.75]{$1$} -- (0.2,1.5);
	\draw (-0.2,3) node[left,scale=.75]{$2$} -- (0.2,3);
	\draw[fill,color=black!50] (1.5,1.5) -- (3.2,1.5) -- (3.2,3.2) -- (1.5,3.2) -- cycle;
	\draw (1.5,1.5) -- (1.5,3.2);
	\draw (1.5,1.5) -- (3.2,1.5);
	\draw (1.5,1.5)[fill,color=black] circle (2pt);
\end{tikzpicture}
\qquad
\qquad
\qquad
\begin{tikzpicture}
	\draw[->] (-0.2,0) -- (3.4,0);
	\draw[->] (0,-0.2) -- (0,3.4);
	\draw (1.5,-0.2) node[below,scale=.75]{$1$} -- (1.5,0.2);
	\draw (3,-0.2) node[below,scale=.75]{$2$} -- (3,0.2);
	\draw (-0.2,1.5) node[left,scale=.75]{$1$} -- (0.2,1.5);
	\draw (-0.2,3) node[left,scale=.75]{$2$} -- (0.2,3);
	\draw[fill,color=black!50] (1.5,1.5) -- (3.2,1.5) -- (3.2,3.2) -- (1.5,3.2) -- cycle;
	\draw[color=white] (1.5,1.5) -- (1.5,3.2);
	\draw[dashed] (1.5,1.5) -- (1.5,3.2);
	\draw[color=white] (1.5,1.5) -- (3.2,1.5);
	\draw[dashed] (1.5,1.5) -- (3.2,1.5);
	\draw (1.5,1.5)[fill,color=white] circle (2pt);
	\draw (1.5,1.5) circle (2pt);
\end{tikzpicture}
\end{center}

\

\noindent have 
$\dist{I_{(1,1),(0,0)}}{I_{(1,1),(1,1)}}=0$
even though $I_{(1,1),(0,0)}\neq I_{(1,1),(1,1)}$.
This explains (at least partially) why we restrict to the set of finitely generated monomial ideals in our next result.
\end{ex}

\begin{thm}\label{tmetric}
 The function $\operatorname{dist}$ is a metric on the set of non-zero finitely generated monomial ideals of $R$.
\begin{proof}
Let $I,J,K$ be non-zero finitely generated monomial ideals of $R$.

To see that $\dist{I}{J} \in\mathbb R_{\geq 0}$, let $\ul X^{\ul\alpha}\in\monset I$ and $\ul X^{\ul\beta}\in\monset J$.
Let $\delta>0$ be given. We claim that
$\dist IJ<\max\{|\ul\alpha|,|\ul\beta|\}+\delta$.
\eqref{def130521a1} For every $\ul X^{\ul\gamma}\in\monset I$ there is a monomial
$\ul X^{\ul\gamma+\ul\beta}=\ul X^{\ul\gamma}\ul X^{\ul\beta}\in\monset{J}$
such that $\dist{\ul\gamma+\ul\beta}{\ul\gamma}=|\ul\beta|<\max\{|\ul\alpha|,|\ul\beta|\}+\delta$.
\eqref{def130521a2} For every $\ul X^{\ul\gamma}\in\monset J$ there is a monomial
$\ul X^{\ul\gamma+\ul\alpha}=\ul X^{\ul\gamma}\ul X^{\ul\alpha}\in\monset{I}$
such that $\dist{\ul\gamma+\ul\alpha}{\ul\gamma}=|\ul\alpha|<\max\{|\ul\alpha|,|\ul\beta|\}+\delta$.
This established the claim. 
Consequently, it follows that $\dist IJ\in\mathbb R$ such that $0\leq\dist IJ\leq\max\{|\ul\alpha|,|\ul\beta|\}$, as desired.

The condition $\dist{I}{J} = \dist{J}{I}$ follows 
from the symmetry of  Definition~\ref{def130521a}.
The equality $\dist{I}{I} = 0$ is similarly straightforward.

Next, assume that $\dist{I}{J} = 0$. 
We show that $I=J$.
Let $\ul{X}^{\ul{\alpha}} \in \monset I$ and let $J=(\ul{X}^{\ul{\beta}_1}, \ldots, \ul{X}^{\ul{\beta}_n})R$. For each $i,j$ we set
$\beta'_{i,j}=\max\{\beta_{i,j},\alpha_j\}$. Then $\beta'_{i,j} \geq \alpha_j$ and $\beta'_{i,j} \geq \beta_{i,j}$. Therefore, for all $j,$ 
we have $\ul{X}^{\ul{\beta}'_j} \in (\mono{X}{\alpha} )R \subseteq I$ and $\ul{X}^{\ul{\beta}'_j} \in (\ul{X}^{\ul{\beta}_j})R \subseteq J$.
Since $\dist{I}{J} = 0,$ we have that $\inf\{\dist{\ul{X}^{\ul{\alpha}}}{\ul{X}^{\ul{\gamma}}} \;|\; \ul{X}^{\ul{\gamma}} \in J\} = 0$. 

We claim that
$$\inf\{\dist{\mono{X}{\alpha}}{\mono{X}{\gamma}} \;|\; \ul{X}^{\ul{\gamma}} \in \monset J\} = \min\{\dist{\mono{X}{\alpha}}{\ul{X}^{\ul{\beta}'_i}} \st 1 \leq i \leq n \}.$$
The condition $\ul{X}^{\ul{\beta}'_j}\in \monset J$ explains the inequality $\leq$. 
For the reverse inequality, let $\ul{X}^{\ul{\gamma}} \in J$. Then $\mono{X}{\gamma} = \mon{X}{\beta}{j}\mono{X}{\delta}$ for some $j$ and some $\mono{X}{\delta} \in \monset R$. 
We have that $\dist{\mono{X}{\alpha}}{\mono{X}{\gamma}} = \sqrt{(\gamma_1-\alpha_1)^2 + \cdots + (\gamma_n - \alpha_n)^2}$. 
If $\gamma_i \leq \alpha_i,$ then we have
$\beta_{j,i} \leq \gamma_i \leq \alpha_i$, so  $\beta_{j,i}' = \alpha_i$, which implies that $| \beta_{j,i}' - \alpha_i | = 0 \leq | \gamma_i - \alpha_i |$. 
If $\gamma_i \geq \alpha_i,$ then the condition $\gamma_i \geq \beta_{j,i}$ implies that
$\gamma_i \geq \beta_{j,i}' \geq \alpha_j,$ so we have that $0 \leq \beta_{j,i}' - \alpha_i \leq \gamma_i - \alpha_i$. Therefore, for all $\mono{X}{\gamma} \in J,$ we have
$\dist{\mono{X}{\gamma}}{\mono{X}{\alpha}} \geq \dist{\ul{X}^{\ul{\beta}_j'}}{\mono{X}{\alpha}}$ for some $j$. This proves the claim.

It follows that $\min\{\dist{\mono{X}{\alpha}}{\ul{X}^{\ul{\beta}'_i}} \st 1 \leq i \leq n \} = 0$. Thus, there exists an index $i$
such that $\mono{X}{\alpha}=\ul{X}^{\ul{\beta}'_i}\in J$. As this is so for all $\ul{X}^{\ul{\alpha}} \in \monset I$, we conclude that $I \subseteq J$. By symmetry,
we have $I = J$, as desired.
  
Now, we check the triangle inequality: $\dist{I}{K} \leq \dist{I}{J} + \dist{J}{K}$. 
Let $\varepsilon > 0$ be given; we  show that $\dist{I}{K} < \dist{I}{J} + \dist{J}{K}+\varepsilon$.
Set $d=\dist{I}{J}$ and $e=\dist{J}{K}$,
and let $\mono{X}{\alpha} \in \monset{I}$. We need to find a monomial $\mono{X}{\gamma} \in \monset{K}$ such that 
$\dist\alpha\gamma<d+e+\varepsilon$.
Since $\dist{I}{J}  < d + \frac{\varepsilon}{2}$, there is a monomial $\mono{X}{\beta} \in \monset{J}$ such that $\dist\alpha\beta<d + \frac{\varepsilon}{2}$.
Similarly, there is a monomial $\mono{X}{\gamma} \in \monset{K}$ such that $\dist\beta\gamma<e + \frac{\varepsilon}{2}$.
From the triangle equality in $\mathbb R^d$, we conclude that
   \begin{align*}
\dist\alpha\gamma\leq\dist\alpha\beta+\dist\beta\gamma<d+e+\varepsilon
    \end{align*}
as desired.
\end{proof}
\end{thm}

 \begin{thm}\label{tsemi-cont}
Given a non-zero finitely generated monomial ideal $I$ in $R$, there exists $\varepsilon > 0$ such that for all non-zero monomial ideals $J,$ with $\dist{I}{J} < \varepsilon,$ we have $\mdim{R/J} \geq \mdim{R/I}$.
  \begin{proof}
  Since $\mdim{R/R}=-\infty$, we assume without loss of generality that $I\neq R$.
    Let $I=(\ul{X}^{\ul{\alpha}_1},\ldots, \ul{X}^{\ul{\alpha}_n})R$, and set $\varepsilon = \min\{\alpha_{ij} \st \alpha_{ij} > 0\} > 0$. Let $J$ be a non-zero monomial ideal with $\dist{I}{J} < \varepsilon$. 
    
  Claim: For every m-prime ideal $Q$ of $R$ such that $Q \supseteq I,$ we have $Q \supseteq J$.
    Let $\mono{X}{\gamma} \in \monset J$ and choose $\mono{X}{\delta} \in I$ such that $\dist{\ul\gamma}{\ul\delta} < \varepsilon$. There exists 
    an index $i$ such that $\mono{X}{\delta}\in(\mon{X}{\alpha}{i})R\subseteq I\subseteq Q$. Since $Q$ is m-prime, there exists 
    an index $j$ such that $\alpha_{i,j} > 0$ and $X_j^{\alpha_{i,j}} \in Q$. Hence, $X_j^t \in Q$ for all $t > 0$. Note that $0 < \varepsilon \leq \alpha_{i,j} \leq \delta_j$
     and $|\gamma_j - \delta_j| \leq\dist{\ul{\gamma}}{\ul{\delta}}<\varepsilon$. Therefore, we have
$- \varepsilon < \gamma_j - \delta_j < \varepsilon$, which implies that
$0 \leq \delta_j - \varepsilon < \gamma_j$, so we conclude that
$\mono{X}{\gamma} \in Q$, which establishes the claim.

It follows $\{ P \supseteq I \st P \text{ is m-prime}\} \subseteq \{ P' \supseteq J \st P' \text{ is m-prime}\}$, and hence $\mdim{R/I} \leq \mdim{R/J}$.
  \end{proof} 
 \end{thm}

 \begin{cor}\label{cor130521a}
  The monomial Krull dimension function is lower semicontinuous on the set of finitely generated monomial ideals of $R$.
 \end{cor}

\begin{ex}\label{ex130508j}
Let $I$ be a non-zero finitely generated monomial ideal of~$R$.

If $\mdim{R/I}=d-1$, then there is a real number $\epsilon>0$ such that
for all monomial ideals $J$ in $R$ such that $\dist IJ<\epsilon$, one has
$\mdim{R/J}=d-1$. Indeed, \tref{semi-cont} provides a real number $\epsilon>0$ such that
for all non-zero monomial ideals $J$ in $R$ such that $\dist IJ<\epsilon$, one has
$\mdim{R/J}\geq d-1$. If $\mdim{R/J}> d-1$, then $\mdim{R/J}=d$, so Example~\ref{ex130509a} implies that
$J=0$, a contradiction.

If  $\mathbb M^i=\mathbb R$ for all $i$, then (regardless of the value of $\mdim{R/I}$) for each real number $\epsilon>0$ there is a non-zero finitely generated
monomial ideal $J\subset I$ in $R$ such that $0<\dist IJ<\epsilon$ and
$\mdim{R/J}=d-1$.
Indeed, consider the ideal $X_1^{\epsilon/2}I$, which is non-zero and finitely generated since $I$ is so.
As in Example~\ref{ex130509c}, it is straightforward to show that $\dist{I}{X_1^{\epsilon/2}I}=\epsilon/2<\epsilon$.
However, the ideal $X_1^{\epsilon/2}I$ is contained in $Q_{\{1\}}=(X_1^a\mid a>0)R$,
so we have $\mdim{X_1^{\epsilon/2}I}\geq d-v(Q_{\{1\}})=d-1$. In particular, if $\mdim{R/I}<d-1$, then strict inequality can occur
in \tref{semi-cont}. This behavior is depicted in the following two diagrams with $d=2$: the first diagram represents  the ideal
$I=J_{(1,1),(0,0)}=(X_1,X_2)R$ and the second one represents 
$X_1^{\epsilon/2}I=(X_1^{1+(\epsilon/2)},X_1^{\epsilon/2}X_2)R=(X_1^{\epsilon/2})R\cap(X_1^{1+(\epsilon/2)},X_2)R$.

\

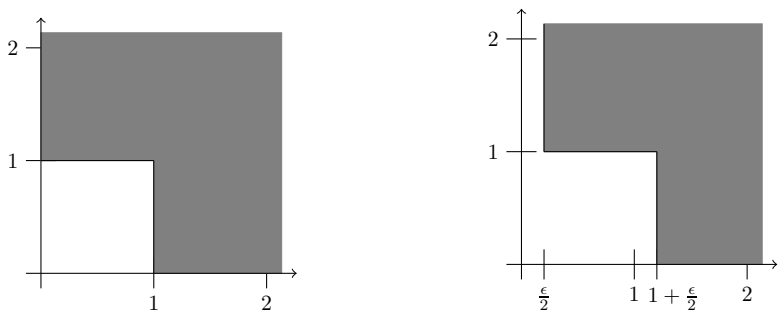
\begin{center}
\begin{tikzpicture}
	\draw (1.5,-0.2) node[below,scale=.75]{$1$} -- (1.5,0.2);
	\draw (3,-0.2) node[below,scale=.75]{$2$} -- (3,0.2);
	\draw (-0.2,1.5) node[left,scale=.75]{$1$} -- (0.2,1.5);
	\draw (-0.2,3) node[left,scale=.75]{$2$} -- (0.2,3);
	\draw[fill,color=black!50] (0,1.5) -- (1.5,1.5) -- (1.5,0) -- (3.2,0) -- (3.2,3.2) -- (0,3.2) -- cycle;
	\draw[->] (-0.2,0) -- (3.4,0);
	\draw[->] (0,-0.2) -- (0,3.4);
	\draw (0,1.5) -- (1.5,1.5);
	\draw (1.5,0) -- (1.5,1.5);
\end{tikzpicture}
\qquad
\qquad
\qquad
\begin{tikzpicture}
	\draw (1.5,-0.2) node[below,scale=.75]{$1$} -- (1.5,0.2);
	\draw (3,-0.2) node[below,scale=.75]{$2$} -- (3,0.2);
	\draw (0.3,-0.2) node[below,scale=.75]{$\frac\epsilon 2$} -- (0.3,0.2);
	\draw (1.8,-0.2) node[below,scale=.75]{\,\,\,\,\,\,\,\,\,$1+\frac\epsilon 2$} -- (1.8,0.2);
	\draw (-0.2,1.5) node[left,scale=.75]{$1$} -- (0.2,1.5);
	\draw (-0.2,3) node[left,scale=.75]{$2$} -- (0.2,3);
	\draw[fill,color=black!50] (0.3,1.5) -- (1.8,1.5) -- (1.8,0) -- (3.2,0) -- (3.2,3.2) -- (0.3,3.2) -- cycle;
	\draw[->] (-0.2,0) -- (3.4,0);
	\draw[->] (0,-0.2) -- (0,3.4);
	\draw (0.3,1.5) -- (0.3,3.2);
	\draw (0.3,1.5) -- (1.8,1.5);
	\draw (1.8,0) -- (1.8,1.5);
\end{tikzpicture}
\end{center}
\end{ex}

\section*{Acknowledgments}
We are grateful Jon Totushek for teaching us how to create the diagrams for  our examples.

\bibliographystyle{plain}

\end{document}